\documentclass[letterpaper,11pt]{amsart}

\usepackage{amsmath,amssymb,amsxtra,amsthm,hyperref}


\newcommand{\bh}[1] {\mathcal{B}(\mathcal{#1})}
\newcommand{\diag} {\operatorname{diag}}

\newtheorem{theorem}{Theorem}[section]
\newtheorem{corollary}[theorem]{Corollary}
\newtheorem{proposition}[theorem]{Proposition}    
\newtheorem{lemma}[theorem]{Lemma}

\theoremstyle{definition}
\newtheorem{definition}[theorem]{Definition}

\newtheorem{example}[theorem]{Example}
\newtheorem{remark}[theorem]{Remark}

\begin{document}

\title{Regular Representations of Lattice Ordered Semigroups}
\author{Boyu Li}
\address{Pure Mathematics Department\\University of Waterloo\\Waterloo, ON\\Canada \ N2L--3G1}
\email{b32li@math.uwaterloo.ca}
\date{\today}

\subjclass[2010]{43A35 ,47A20 ,47D03 }
\keywords{ Nica covariant, regular dilation, positive definite, lattice ordered group}

\begin{abstract} We establish a necessary and sufficient condition for a representation of a lattice ordered semigroup to be regular, in the sense that certain extensions are completely positive definite. This result generalizes a theorem due to Brehmer where the lattice ordered group was taken to be $\mathbb{Z}_+^\Omega$. As an immediate consequence, we prove that contractive Nica-covariant representations are regular. We also introduce an analog of commuting row contractions on lattice ordered group and show that such a representation is regular. 
\end{abstract}

\maketitle

\section{Introduction}

A contractive map of a group has a unitary dilation if and only if it is completely positive definite, in the sense that certain operator matrices are positive. Consequently, for a semigroup $P$ contained in a group $G$, a contractive representation of $P$ has a unitary dilation if and only if it can be extended to a completely positive definite map on $G$. Introduced in \cite{DFK2014}, such representations on a semigroup are called completely positive definite. In particular, when the group is lattice-ordered, a representation is called regular if a certain natural extension to the group is completely positive definite. 

Nica \cite{Nica1992} introduced the study of isometric representations of quasi-lattice ordered semigroups. This generalized the notion of doubly commuting representations of semigroups with nice generators. Laca and Raeburn \cite{Laca1996} developed the theory, and showed there is a universal $C^*$-algebra for isometric Nica covariant representations. This field has also been explored in \cite{Solel2008}.

Davidson, Fuller, and Kakariadis \cite{Fuller2013, DFK2014} defined and studied contractive Nica-covariant representation on lattice ordered semigroups. The regularity of such representations was seen as a critical property in describing the $C^*$-envelope of semicrossed products. They posed a question \cite[Question 2.5.11]{DFK2014} of whether regularity is automatic for Nica-covariant representations. Fuller \cite{Fuller2013} established this for certain abelian semigroups. 

This paper answers this question affirmatively by establishing a necessary and sufficient condition for a representation of a lattice ordered semigroup to be regular. This condition generalizes a result of Brehmer \cite{Brehmer1961}, where he gave a necessary and sufficient condition for a representation of $\mathbb{Z}_+^\Omega$ to be regular. As an immediate consequence of Brehmer's condition, it is known that doubly commuting representations and commuting column contractions are both regular \cite[Proposition I.9.2]{SFBook}. This paper generalizes both results in the lattice ordered group settings. We first show that a Nica-covariant representation, which is an analog of a doubly commuting representation, is regular. We then introduce an analog of commuting column contractions, which is shown to be regular as well. 

\section{Preliminaries}

Let $G$ be a group. A unital semigroup $P\subseteq G$ is called a cone. A cone $P$ is \emph{spanning} if $PP^{-1}=G$, and is \emph{positive} when $P\bigcap P^{-1}=\{e\}$. A positive cone $P$ defines a partial order on $G$ via $x\leq y$ when $x^{-1}y\in P$. $(G,P)$ is called \emph{totally ordered} if $G=P\bigcup P^{-1}$, in which case the partial order on $G$ is a total order. If any finite subset of $G$ with a upper bound in $P$ also has a least upper bound in $P$, the pair $(G,P)$ is called a \emph{quasi-lattice ordered group}. We call this partial order \emph{compatible with the group} if for any $x\leq y$ and $g\in G$, we always have $gx\leq gy$ and $xg\leq yg$. Equivalently, the corresponding positive cone satisfies a normality condition that $gPg^{-1}\subseteq P$ for any $g\in G$, and thus $x\leq y$ whenever $yx^{-1}\in P$ as well. When $P$ is a positive spanning cone of $G$ whose partial order is compatible with the group, if every two elements $x,y\in G$ have a least upper bound (denoted by $x\vee y$) and a greatest lower bound (denoted by $x\wedge y$), the pair $(G,P)$ is called a \emph{lattice ordered group}. It is immediate that a lattice ordered group is also a quasi-lattice ordered group. 

\begin{example} (Examples of Lattice Ordered Groups)
\begin{enumerate}
\item $(\mathbb{Z},\mathbb{Z}_{\geq 0})$ is a lattice ordered group. In fact, this partial order is also a total order. More generally, any totally ordered group $(G,P)$ is also a lattice ordered group.
\item If $(G_i, P_i)_{i\in I}$ is a family of lattice ordered group, their direct product $(\prod G_i, \prod P_i)$ is also a lattice ordered group.
\item Let $G=C_\mathbb{R}[0,1]$, the set of all continuous functions on $[0,1]$. Let $P$ be the set of all non-negative functions in $G$. Then $(G,P)$ is a lattice ordered group.
\item Let $\mathcal{T}$ be a totally ordered set. A permutation $\alpha$ on $\mathcal{T}$ is called order preserving if for any $p,q\in\mathcal{T}$, $p\leq q$, we also have $\alpha(p)\leq\alpha(q)$. Let $G$ be the set of all order preserving permutations, which is clearly a group under composition. Let $P=\{\alpha\in G: \alpha(t)\geq t,\mbox{ for all }t\in\mathcal{T}\}$. Then $(G,P)$ is a non-abelian lattice ordered group \cite{LatticeOrderBookIntro}.
\item Let $\mathbb{F}_n$ be the free group of $n$ generators, and $\mathbb{F}_n^+$ be the semigroup generated by the $n$-generators. Then $(\mathbb{F}_n,\mathbb{F}_n^+)$ defines a quasi-lattice ordered group \cite[Examples 2.3]{Nica1992}. However, the induced partial order is not compatible with the group and the pair is not a lattice ordered group. 
\end{enumerate}
\end{example}

For any element $g\in G$ of a lattice ordered group $(G,P)$, $g$ can be written uniquely as $g=g_+g_-^{-1}$ where $g_+,g_-\in P$, and $g_+\wedge g_-=e$. In fact, $g_+=g\vee e$ and $g_-=g^{-1}\vee e$. Here are some important properties of a lattice ordered group:

\begin{lemma}\label{lm.LatticeBasic} Let $(G,P)$ be a lattice order group, and $a,b,c\in G$.
\begin{enumerate}
\item $a(b\vee c)=(ab)\vee(ac)$ and $(b\vee c)a=(ba)\vee(ca)$. A similar distributive law holds for $\wedge$.
\item $(a\wedge b)^{-1}=a^{-1}\vee b^{-1}$ and similarly $(a\vee b)^{-1}=a^{-1}\wedge b^{-1}$.
\item $a\geq b$ if and only if $a^{-1}\leq b^{-1}$.
\item $a(a\wedge b)^{-1} b=a\vee b$. In particular, when $a\wedge b=e$, $ab=ba=a\vee b$. 
\item If $a,b,c\in P$, then $a\wedge(bc)\leq (a\wedge b)(a\wedge c)$.
\end{enumerate}
\end{lemma}

One may refer to \cite{LatticeOrderBook} for a detailed discussion of this subject. Notice by statement (4) of Lemma \ref{lm.LatticeBasic} $g_+,g_-$ commute and thus $g=g_+g_-^{-1}=g_-^{-1} g_+$.

For a group $G$, a unital map $S:G\to\bh{H}$ is called \emph{completely positive definite} if for any $g_1,g_2,\cdots,g_n\in G$
$$\big[S(g_i^{-1} g_j)\big]_{1\leq i,j\leq n}\geq 0.$$
Here, $i$ denotes the row index and $j$ the column index, and we shall follow this convention throughout this paper. A well known result (\cite{Naimark1943}, see also \cite[Proposition I.7.1]{SFBook}) stated that a completely positive definite map of $G$ has a unitary dilation. The converse is elementary. 

\begin{theorem}\label{thm.NagyFoias} If $S:G\to\bh{H}$ is a unital completely positive definite map. Then there exists a unitary representation $U:G\to\bh{K}$ where $\mathcal{H}$ is a subspace of $\mathcal{K}$, and that $P_\mathcal{H} U(g)|_\mathcal{H} = S(g)$. Moreover, this unitary representation can be chosen to be minimal in the sense of $\mathcal{K}=\bigvee_{g\in G} U(g)\mathcal{H}$.
\end{theorem}

When $(G,P)$ is a lattice ordered group, we may simultaneously increase or decrease $g_i$ so that it would suffices to take $g_i\in P$:

\begin{lemma} Let $S:G\to\bh{H}$ be a map, then the following are equivalent:
\begin{enumerate}
\item $\big[S(g_i^{-1} g_j)\big]_{1\leq i,j\leq n} \geq 0$ for any $g_1,g_2,\cdots,g_n\in G$;
\item $\big[S(g_i g_j^{-1})\big]_{1\leq i,j\leq n} \geq 0$ for any $g_1,g_2,\cdots,g_n\in G$;
\item $\big[S(p_i^{-1} p_j)\big]_{1\leq i,j\leq n} \geq 0$ for any $p_1,p_2,\cdots,p_n\in P$;
\item $\big[S(p_i p_j^{-1})\big]_{1\leq i,j\leq n} \geq 0$ for any $p_1,p_2,\cdots,p_n\in P$.
\end{enumerate}
\end{lemma}

\begin{proof}

Since $G$ is a group, by considering $g_i$ and $g_i^{-1}$, it is clear that (1) and (2) are equivalent. Statement (1) clearly implies statement (3), and conversely when statement (3) holds true, for any $g_1,\cdots,g_n\in G$, take $g=\vee_{i=1}^n \left(g_i\right)_-$. Denote $p_i=g\cdot g_i$ and notice that from our choice of $g$, $g\geq \left(g_i\right)_-$. Hence, $$p_i=g\cdot \left(g_i\right)_-^{-1} \left(g_i\right)_+\in P.$$
But notice that for each $i,j$, $p_i^{-1} p_j=g_i^{-1} g^{-1} gg_j=g_i^{-1}g_j$. Therefore, $$\big[S(g_i^{-1} g_j)\big]_{1\leq i,j\leq n}=\big[S(p_i^{-1} p_j)\big]_{1\leq i,j\leq n}\geq 0.$$
Similarly, statements (2) and (4) are equivalent. \end{proof}

For the convenience of computation, when $(G,P)$ is a lattice ordered group, $S:G\to\bh{H}$ is called completely positive definite when $$\big[S(p_i p_j^{-1})\big]_{1\leq i,j\leq n} \geq 0.$$
For a spanning cone $P\subset G$, a contractive representation $T:P\to\bh{H}$ is called \emph{completely positive definite} when it can be extended to some completely positive definite map on $G$. There is a well-known result due to Sz.Nagy that every contraction has a unitary dilation, and therefore, every contractive representation of $\mathbb{Z}_+$ is completely positive definite. Ando \cite{Ando1963} further showed that every contractive representation of $\mathbb{Z}_+^2$ is completely positive definite. However, Parrott \cite{Parrott1970} provided an counterexample where a contractive representations on $\mathbb{Z}_+^3$ is not completely positive definite. 

For a completely positive definite representation $T$ on a lattice ordered semigroup, one might wonder what its extension looks like. In a lattice ordered group $(G,P)$, any element $g\in G$ can be uniquely written as $g=g_+g_-^{-1}$ where $g_\pm\in P$ and $g_+\wedge g_-=e$. Suppose $U:G\to\bh{K}$ is a unitary dilation of $T$, we can make the following observation.
\begin{eqnarray*}
\tilde{T}(g) &=& P_\mathcal{H} U(g)\big|_\mathcal{H} \\
&=& P_\mathcal{H} U(g_-)^*U(g_+)\big|_\mathcal{H}.
\end{eqnarray*}
This motivates the question of whether the extension $\tilde{T}(g)=T(g_-)^*T(g_+)$ is completely positive definite. We call a contractive representation $T$ \emph{right regular} whenever $\tilde{T}$ defined in such way is completely positive definite. There is a dual definition that call $T$ \emph{left regular} (or \emph{$*$-regular}) if $\overline{T}(g)=T(g_+)T(g_-)^*$ is completely positive definite.

When $(G,P)$ is a lattice ordered group, $(G,P^{-1})$ is also a lattice ordered group. A representation $T:P\to\bh{H}$ give raise to a dual representation $T^*:P^{-1}\to\bh{H}$ where $T^*(p^{-1})=T(p)^*$. Consider $g=g_+g_-^{-1}=g_-^{-1} \left(g_+^{-1}\right)^{-1}$, we have $$\tilde{T}(g)=T(g_-)^* T(g_+)=T^*(g_-^{-1}) T^*(g_+^{-1})^*=\overline{T^*}(g).$$ Hence, $\tilde{T}$ agrees with $\overline{T^*}$ on $G$. Therefore, we obtain the following Proposition.

\begin{proposition} Let $(G,P)$ be a lattice ordered group, and $T:P\to\bh{H}$ be a representation and $T^*$ defined as above. Then the following are equivalent
\begin{enumerate}
\item $T$ is right regular.
\item $T^*$ is left regular.
\item For any $p_1,\cdots,p_n\in P$, $[\tilde{T}(p_ip_j^{-1})]\geq 0$ (equivalently, $[\overline{T^*}(p_ip_j^{-1})]\geq 0$).
\end{enumerate}
\end{proposition}

Due to this equivalence, we shall focus on the right regularity and call a representation \emph{regular} when it is right regular. Regular dilations were first studied by Brehmer \cite{Brehmer1961}, and they were also studied in \cite{Nagy1961, Halperin1962}. A necessary and sufficient condition for regularity for the abelian group $\mathbb{Z}^\Omega$ was proven by Brehmer \cite[Theorem I.9.1]{SFBook}.

\begin{theorem}[Brehmer]\label{thm.Brehmer} Let $\Omega$ be a set, and denote $\mathbb{Z}^\Omega$ to be the set of $(t_\omega)_{\omega\in\Omega}$ where $t_\omega\in\mathbb{Z}$ and $t_\omega=0$ except for finitely many $\omega$. Also, for a finite set $V\subset\Omega$, denote $e_V\in\mathbb{Z}^\Omega$ to be $1$ at those $\omega\in V$ and $0$ elsewhere. If $\{T_\omega\}_{\omega\in\Omega}$ is a family of commuting contractions, we may define a contractive representation $T:\mathbb{Z}_+^\Omega\to\bh{H}$ by $$T(t_\omega)_{\omega\in\Omega}=\prod_{\omega\in\Omega} T_\omega^{t_\omega}.$$
Then $T$ is right regular if and only if for any finite $U\subseteq\Omega$, the operator $$\sum_{V\subseteq U} (-1)^{|V|} T(e_V)^* T(e_V)\geq 0.$$
\end{theorem}

It turns out that not all completely positive definite representations are regular.

\begin{example} It follows from Brehmer's theorem that a representation $T$ on $\mathbb{Z}_+^2$ is regular if and only if $T_1=T(e_1), T_2=T(e_2)$ are contractions that satisfy $$I-T_1^* T_1-T_2^* T_2+(T_1 T_2)^* T_1T_2\geq 0.$$
Take $T_1=T_2=\left[ \begin{array}{cc}
0 & 1 \\
0 & 0 \\ \end{array} \right]$ and notice, $$I-T_1^* T_1-T_2^* T_2+(T_1 T_2)^* T_1T_2=\left[ \begin{array}{cc}
1 & 0 \\
0 & -1 \\ \end{array} \right].$$

Brehmer's result implies that $T$ is not regular. However, from Ando's theorem \cite{Ando1963}, any contractive representation on $\mathbb{Z}_+^2$ has a unitary dilation and thus is completely definite definite. 
\end{example}

Isometric Nica-covariant representations on quasi-lattice ordered groups were first introduced by Nica \cite{Nica1992}: an isometric representation $W:G\to\bh{H}$ is Nica-covariant if for any $x,y$ with an upper bound, $W_x W_x^* W_y W_y^*=W_{x\vee y} W_{x\vee y}^*$. When the order is a lattice order, it is equivalent to the property that $W_s,W_t^*$ commute whenever $s\wedge t=e$. Therefore, the notion of Nica-covariant is extended to abelian lattice ordered groups in \cite{DFK2014}, and we shall further extend such definition to non-abelian lattice ordered groups and call a representation $T:P\to\bh{H}$ \emph{Nica-covariant} if $T(s)T(t)^*=T(t)^*T(s)$ whenever $s\wedge t=e$. For a Nica-covariant representation $T$, since $T(g^+)$ commutes with $T(g^-)^*$ for any $g\in G$, there is no difference between left and right regularity. It observed in \cite{DFK2014} that Nica-covariant representations are regular in many cases. 

\begin{example} (Examples of Nica covariant representations)
\begin{enumerate}
\item On $(\mathbb{Z},\mathbb{Z}_+)$, a contractive representation $T$ on $\mathbb{Z}_+$ only depends on $T_1=T(1)$ since $T(n)=T_1^n$. This representation is always Nica-covariant since for any $s,t\geq 0$, $s\wedge t=0$ if and only if one of $s,t$ is $0$. A well known result due to Sz.Nagy shows that its extension to $\mathbb{Z}$ by $\tilde{T}(-n)=T^{*n}$ is completely positive definite and thus $T$ is regular.
\item Similarly, any contractive representation of a totally ordered group $(G,P)$ is Nica-covariant. A theorem of Mlak \cite{Mlak1966} shows that such representations are regular.
\item $(\mathbb{Z}^n,\mathbb{Z}_+^n)$, the finite Cartesian product of $(\mathbb{Z},\mathbb{Z}_+)$ is a lattice ordered group. A representation $T$ on $\mathbb{Z}_+^n$ depends on $n$ contractions $T_1=T(1,0,\cdots,0)$, $T_2=T(0,1,0,\cdots,0)$,$\cdots$, $T_n=T(0,\cdots,0,1)$. Notice $T$ is Nica covariant if and only if $T_i, T_j$ $*$-commute whenever $i\neq j$. Hence Nica covariant representations are equivalent to doubly commuting. It is known \cite[Section I.9]{SFBook} that doubly commuting contractive representations are regular. 
\item For a lattice ordered group made from a direct product of totally ordered groups, Fuller \cite{Fuller2013} showed that their contractive Nica-covariant representations are regular.
\end{enumerate}
\end{example}

A question posed in \cite[Question 2.5.11]{DFK2014} asks whether contractive Nica-covariant representations on abelian lattice ordered groups are regular in general. For example, for $G=C_\mathbb{R}[0,1]$ and $P$ equal to the set of non-negative continuous functions, there are no known results on whether contractive Nica-covariant representations are regular on such semigroup. Little is known for the non-abelian lattice ordered groups as well. In this paper, we establish that all Nica-covariant representations of lattice ordered semigroups are regular.

Let $(G,P)$ be a lattice-ordered group, not necessarily abelian. Recall that the regularity conditions require a matrix involving entries in the form of $\tilde{T}(pq^{-1})$ to be positive, where $p,q\in P$. We start by investigating this quantity of $pq^{-1}$.

\begin{lemma}\label{lm.1} Let $p,q\in P$. Then,
\begin{eqnarray*}
(pq^{-1})_+ &=& p(p\wedge q)^{-1}\mbox{  and,}\\
(pq^{-1})_- &=& q(p\wedge q)^{-1}.
\end{eqnarray*} \end{lemma}

\begin{proof}

By property (1) and (2) in Lemma \ref{lm.LatticeBasic}, 
\begin{eqnarray*}
(pq^{-1})_+ &=& (pq^{-1} \vee e) \\
&=& p(q^{-1} \vee p^{-1}) \\
&=& p(p\wedge q)^{-1}.
\end{eqnarray*}
Similarly, $(pq^{-1})_-=q(p\wedge q)^{-1}.$   \end{proof}

\begin{lemma}\label{lm.1b} Let $p,q,g\in P$ such that $g\wedge q=e$. Then $(pg)\wedge q=p\wedge q$.
\end{lemma}
\begin{proof}

By the property (5) of Lemma \ref{lm.LatticeBasic}, we have that
$$(pg)\wedge q\leq (p\wedge q)(g\wedge q)=p\wedge q.$$
On the other hand, $p\wedge q$ is clearly a lower bound for both $p\leq pg$ and $q$, and hence $p\wedge q\leq (pg)\wedge q$. This proves the equality. \end{proof}

\begin{lemma}\label{lm.2} Let $p,q\in P$. If $g\in P$ is another element where $g\wedge q=0$, then 
\begin{eqnarray*}
(pgq^{-1})_- &=& (pq^{-1})_-\mbox{  and,} \\
(pgq^{-1})_+ &=& (pq^{-1})_+g. 
\end{eqnarray*}
In particular, if $0\leq g\leq p$, then
\begin{eqnarray*}
(pg^{-1}q^{-1})_- &=& (pq^{-1})_-\mbox{  and,} \\
(pg^{-1}q^{-1})_+ &=& (pq^{-1})_+g^{-1}.
\end{eqnarray*}
\end{lemma}  
\begin{proof}

By Lemma \ref{lm.1}, we get  $(pgq^{-1})_+=pg(q\wedge pg)^{-1}$. Apply Lemma \ref{lm.1b} to get
$$(q\wedge pg)^{-1} = (q\wedge p)^{-1}.$$

Now $g\wedge (p\wedge q)=e$ and thus $g$ commutes with $p\wedge q$ by property (4) of Lemma \ref{lm.LatticeBasic}. Therefore,
\begin{eqnarray*}
(pgq^{-1})_+ &=& pg(q\wedge pg)^{-1} \\
&=& p(q\wedge p)^{-1}g  \\
&=& (pq^{-1})_+ g.
\end{eqnarray*}
The statement $(pgq^{-1})_- = (pq^{-1})_-g$ can be proven in a similar way. 

Finally, for the case where $0\leq g\leq p$, it follows immediately by considering $p'=pg^{-1}$ and thus $p=p'g$. \end{proof}

\begin{lemma}\label{lm.order} If $p_1,p_2,\cdots,p_n\in P$ and $g_1,\cdots,g_n\in P$ be such that $g_i\leq p_i$ for all $i=1,2,\cdots,n$. Then $\wedge_{i=1}^n p_i g_i^{-1} \leq \wedge_{i=1}^n p_i$. In particular, when $\wedge_{i=1}^n p_i=e$, we have $\wedge_{i=1}^n p_i g_i^{-1}=e$.
\end{lemma}

\begin{proof} It is clear that $e\leq p_i g_i^{-1}\leq p_i$, and thus $$e\leq \wedge_{i=1}^n p_i g_i^{-1} \leq \wedge_{i=1}^n p_i.$$ Therefore, the equality holds when the later is $e$. \end{proof}

\section{A Necessary and Sufficient Condition For Regularity}

When $T:P\to\bh{H}$ is a representation of lattice ordered semigroup, we denote $\tilde{T}(g)=T(g^-)^* T(g^+)$. Recall that $T$ is \emph{regular} if $\tilde{T}$ is completely positive definite. The main result is the following necessary and sufficient condition for regularity:

\begin{theorem}\label{thm.main} Let $(G,P)$ be a lattice ordered group and $T:P\to\bh{H}$ be a contractive representation. Then $T$ is regular if and only if for any $p_1,\cdots,p_n\in P$ and $g\in P$ where $g\wedge p_i=e$ for all $i=1,2,\cdots,n$, we have
\begin{equation}\label{eq.right}\left[T(g)^* \tilde{T}(p_i p_j^{-1}) T(g)\right]\leq \left[\tilde{T}(p_i p_j^{-1})\right].\tag{$\star$}\end{equation}
\end{theorem}

\begin{remark}\label{rm.pos} If we denote $$X=\left[\tilde{T}(p_i p_j^{-1})\right]$$ and $D=
\diag(T(g),T(g),\cdots,T(g))$, Condition (\ref{eq.right}) is equivalent of saying $D^* X D\leq X$. Notice that we made no assumption on $X\geq 0$. Indeed, it follows from the main result that Condition (\ref{eq.right}) is equivalent of saying the representation $T$ is regular, which in turn implies $X\geq 0$. Therefore, when checking the Condition (\ref{eq.right}), we may assume $X\geq 0$. 
\end{remark}

\begin{remark} By setting $p_1=e$ and picking any $g\in P$, Condition (\ref{eq.right}) implies that $T(g)^* T(g)\leq I$, and thus $T$ must be contractive. 
\end{remark}

The following Lemma is taken from \cite[Lemma 14.13]{NestAlgebra}.

\begin{lemma}\label{lm.Davidson} If $A,X,D$ are operators in $\bh{H}$ where $A\geq 0$. Then a matrix of the form $\begin{bmatrix} A & A^{1/2}X \\ X^* A^{1/2} & D\end{bmatrix}$ is positive if and only if $D\geq X^* X$.
\end{lemma}

Condition (\ref{eq.right}) can thus be interpreted in the following equivalent definition.

\begin{lemma}\label{lm.equiv} Let $p_1,\cdots,p_n\in P$ and $g\in P$ with $g\wedge p_i=e$ for all $1\leq i\leq n$. Denote $q_1=p_1g,\cdots,q_n=p_ng$ and $q_{n+1}=p_1,\cdots,q_{2n}=p_n$. Then Condition (\ref{eq.right}) is equivalent to $\left[\tilde{T}(q_i q_j^{-1})\right]\geq 0$.
\end{lemma}

\begin{proof}

Let $X=[\tilde{T}(p_i p_j^{-1})]\geq 0$ and $D=\diag(T(g),T(g),\cdots,T(g))$. Notice by Lemma \ref{lm.2} that 
\begin{eqnarray*}
(p_i g p_j^{-1})_+&=&(p_i p_j^{-1})_+ g\\
(p_i g p_j^{-1})_-&=&(p_i p_j^{-1})_-,
\end{eqnarray*}
and thus $\tilde{T}(p_i g p_j^{-1})=\tilde{T}(p_i p_j^{-1}) T(g)$. Therefore,
$$\left[\tilde{T}(q_i q_j^{-1})\right]=\begin{bmatrix} X & XD \\ D^* X & X\end{bmatrix}.$$ Lemma \ref{lm.Davidson} implies that this matrix is positive if and only if $D^* X D\leq X$, which is Condition \ref{eq.right}.\end{proof}

We shall first show that $\left[\tilde{T}(p_i p_j^{-1})\right]\geq 0$ given $p_i\wedge p_j=e$ and Condition (\ref{eq.right}). This will serve as a base case in the proof of the main result.

\begin{lemma}\label{lm.base} Let $(G,P)$ be a lattice ordered group, and $T$ be a representation on $P$ that satisfies Condition (\ref{eq.right}). If $p_i\wedge p_j=e$ for all $i\neq j$, then $[\tilde{T}(p_i p_j^{-1})]\geq 0$. 
\end{lemma}

\begin{proof}

Let $q_1=e,q_2=p_1$ and for each $1<m\leq n$, recursively define $q_{2^{m-1}+k}=p_m q_k$ where $1\leq k\leq 2^{m-1}$. Since $T$ is contractive, 
$$[\tilde{T}(q_i q_j^{-1})]_{1\leq i,j\leq 2}=\begin{bmatrix} I & \tilde{T}(q_1 q_2^{-1}) \\ \tilde{T}(q_2q_1^{-1}) & I \end{bmatrix}\geq 0.$$
By Lemma \ref{lm.equiv}, for each $m$, $[\tilde{T}(q_i q_j^{-1})]_{1\leq i,j\leq 2^m}\geq 0$. Notice that $q_{2^{m-1}}=p_m$ for each $1\leq m\leq n$. Therefore, $[\tilde{T}(p_i p_j^{-1})]$ is a corner of $[\tilde{T}(q_i q_j^{-1})]\geq 0$, and thus must be positive.\end{proof}

For arbitrary choices of $p_1,\cdots,p_n\in P$, the goal is to reduce it to the case where $p_i\wedge p_j=e$. The following lemma does the reduction.

\begin{lemma}\label{lm.reduce} Let $(G,P)$ be a lattice ordered group. Assuming $T$ is a representation that satisfies Condition (\ref{eq.right}). 

Assume there exists $2\leq k<n$ where for each $J\subset\{1,2,\cdots,n\}$ with $|J|>k$, $\wedge_{j\in J} p_j=e$. Then let $g=\wedge_{j=1}^{k} p_j$ and $q_1=p_1 g^{-1},\cdots,q_k=p_k g^{-1}$, and $q_{k+1}=p_{k+1},\cdots,q_n=p_n$. 
Then $[\tilde{T}(p_i p_j^{-1})]\geq 0$ if $[\tilde{T}(q_i q_j^{-1})]\geq 0$.
\end{lemma}

\begin{proof}

Let us denote $X=[\tilde{T}(q_jq_i^{-1})]\geq 0$ and its lower right $(n-k)\times (n-k)$ corner to be $Y$. Notice first of all, when $i,j\in\{1,2,\cdots,k\}$, $$q_i q_{j}^{-1}=p_i g^{-1} g p_{j}^{-1}=p_ip_{j}^{-1}.$$ 
So the upper left $k\times k$ corner of $[\tilde{T}(q_i q_j^{-1})]$ and the lower right $(n-k)\times(n-k)$ corner of $X$ are both the same as those in $[\tilde{T}(p_i p_j^{-1})]$. 

Now consider $i\in\{1,2,\cdots,k\}$ and $j\in\{k+1,\cdots,n\}$. It follows from the assumption that  $g\wedge p_j = \left(\wedge_{s=1}^k p_s\right)\wedge p_j=e$ and $g\leq p_i$. Therefore, we can apply Lemma \ref{lm.2} to get 
\begin{eqnarray*}
(p_i g^{-1}p_j^{-1})_- &=& (p_ip_j^{-1})_- \\
(p_i g^{-1}p_j^{-1})_+ &=& (p_ip_j^{-1})_+g^{-1}.
\end{eqnarray*}
Now $g\in P$, so that $T((q_i q_j^{-1})_+)T(g)=T((p_i p_j^{-1})_+)$ and $T((q_i q_j^{-1})_-)=T((p_i p_j^{-1})_-)$. Hence, $$\tilde{T}(q_i q_j^{-1}) T(g)=\tilde{T}(p_i p_j^{-1}).$$
Similarly, for $i\in\{k+1,\cdots,n\}$, $j\in\{1,2,\cdots,k\}$, we have $\tilde{T}(p_i p_j^{-1})=T(g)^*\tilde{T}(q_j q_i^{-1})$. Now define $D=\operatorname{diag}(I,\cdots, I, T(g), \cdots, T(g))$ be the block diagonal matrix with $k$ copies of $I$ followed by $n-k$ copies of $T(g)$. Consider $D X D^*$: it follows immediately from the assumption that $D^* X D\geq 0$. We have,
$$D^*[\tilde{T}(q_i q_j^{-1})] D= \renewcommand\arraystretch{1.4}\left[ \begin{array}{ccc|c}
\cdots & \cdots & \cdots & \vdots \\
\cdots & \tilde{T}(p_i p_j^{-1}) & \cdots & \tilde{T}(q_i q_j^{-1}) T(g) \\
\cdots & \cdots & \cdots & \vdots \\ \hline
\cdots & T(g)^* \tilde{T}(q_i q_j^{-1}) & \cdots & [T(g)^* \tilde{T}(p_i p_j^{-1}) T(g)] \\ \end{array} \right]\geq 0.$$
It follows from our previous computation that each entry in the lower left $(n-k)\times k$ corner and upper right $k\times (n-k)$ corner are the same as those in $[\tilde{T}(p_i p_j^{-1})]$. Hence, $D X D^*$ only differs from $[\tilde{T}(p_ip_j^{-1})]$ on the lower right $(n-k)\times (n-k)$ corner. It follows from Condition (\ref{eq.right}) that 
$$[T(g)^* \tilde{T}(p_i p_j^{-1}) T(g)]\leq [\tilde{T}(p_i p_j^{-1})].$$
Hence, the matrix remains positive when the lower right corner $[T(g)^* \tilde{T}(p_i p_j^{-1}) T(g)]$ in $D^* X D$ is replaced by $[\tilde{T}(p_i p_{j}^{-1})]$. The resulting matrix is exactly $[\tilde{T}(p_ip_j^{-1})]$, which must be positive.
\end{proof}

Now the main result (Theorem \ref{thm.main}) can be deduced inductively:

\begin{proof}

First assume that $T:P\to\bh{H}$ is a representation that satisfies Condition (\ref{eq.right}), which has to be contractive. The goal is to show for any $n$ elements $p_1,p_2,\cdots,p_n\in P$, the operator matrix $[\tilde{T}(p_i p_j^{-1})]\geq 0$ and thus $T$ is regular. We proceed by induction on $n$.

For $n=1$, $\tilde{T}(p_1 p_1^{-1})=I\geq 0$.

For $n=2$, we have,
$$[\tilde{T}(p_i p_j^{-1})]= \left[ \begin{array}{cc}
I & \tilde{T}(p_1 p_2^{-1}) \\
\tilde{T}(p_2 p_1^{-1}) & I \end{array} \right].$$
Here, $\tilde{T}(p_2 p_1^{-1}) =  \tilde{T}(p_1 p_2^{-1})^*$, and they are contractions since $T$ is contractive. Therefore, this $2\times 2$ operator matrix is positive.

Now assume that there is an $N$ such that for any $n<N$, we have $[\tilde{T}(p_i p_j^{-1})]\geq 0$ for any $p_1,p_2,\cdots,p_n\in P$. Consider the case when $n=N$:

For arbitrary choices $p_1,\cdots,p_N\in P$, let $g=\wedge_{i=1}^N p_i$, and replace $p_i$ by $p_i g^{-1}$. By doing so, $ p_i g^{-1} \left(p_j g^{-1}\right)^{-1}=p_i p_j^{-1}$, and thus they give the same matrix $[\tilde{T}(p_i p_j^{-1})]$. Moreover, $\wedge_{i=1}^n p_ig^{-1}=(\wedge_{i=1}^N p_i)g^{-1}=e$. Hence, without loss of generality, we may assume $\wedge_{i=1}^N p_i=e$.

Let $m$ be the smallest integer such that for all $J\subseteq \{1,2,\cdots, N\}$ and $|J|>m$, we have $\wedge_{j\in J} p_j=e$. It is clear that $m\leq N-1$. Now do induction on $m$:

For the base case when $m=1$, we have $p_i\wedge p_j=e$ for all $i\neq j$. Lemma \ref{lm.base} tells that Condition (\ref{eq.right}) implies $[\tilde{T}(p_i p_j^{-1})]\geq 0$.

Now assume $[\tilde{T}(p_i p_j^{-1})]\geq 0$ whenever $m\leq M-1<N-1$ and consider the case when $m=M$: For a subset $J\subseteq \{1,2,\cdots,n\}$ with $|J|=M$, let $g=\wedge_{j\in J} p_j$ and set $q_j=p_j g^{-1}$ for all $j\in J$, and $q_j=p_j$ otherwise. Lemma \ref{lm.reduce} concluded that $[\tilde{T}(p_i p_j^{-1})]\geq 0$ whenever $[\tilde{T}(q_i q_j^{-1})]\geq 0$ and the sub-matrix $[\tilde{T}(p_i p_j^{-1})]_{i,j\notin J}\geq 0$.

Since $|\{1,2,\cdots,N\}\backslash J|=N-M<N$, the induction hypothesis on $n$ implies that $[\tilde{T}(p_i p_j^{-1})]_{i,j\notin J}\geq 0$. Therefore, $[\tilde{T}(p_ip_j^{-1})]\geq 0$ whenever $[\tilde{T}(q_i q_j^{-1})]\geq 0$, and by dropping from $p_i$ to $q_i$, we may, without loss of generality, assume that $\wedge_{j\in J} p_j=e$. Repeat this process for all subsets $J\subset\{1,2,\cdots,n\}$ where $|J|=M$, and with Lemma \ref{lm.order}, we eventually reach a state when $\wedge_{j\in J} p_j=e$ for all $J\subseteq \{1,2,\cdots,N\}$, $|J|=M$. But in such case, for all $|J|\geq M$, we have $\wedge_{j\in J} p_j=e$. Therefore, we are in a situation where $m\leq M-1$. The result follows from the induction hypothesis on $m$.

Conversely, suppose that $T$ is regular. Fix $g\in P$ and $p_1,p_2,\cdots,p_k\in P$ where $g\wedge p_i=e$ for all $i=1,2,\cdots,k$. Denote $q_1=p_1g,q_2=p_2g,\cdots,q_k=p_kg$, and $q_{k+1}=p_1,q_{k+2}=p_2,\cdots,q_{2k}=p_k$. It follows from regularity that $[\tilde{T}(q_iq_j^{-1})]\geq 0$, which is equivalent to Condition (\ref{eq.right}) by Lemma \ref{lm.equiv}. \end{proof}

As an immediate consequence of Theorem \ref{thm.main}, we can show that isometric representations on any lattice ordered group must be regular.

\begin{corollary}\label{cor.iso} Let $T:P\to\bh{H}$ be an isometric representation of a lattice ordered semigroup. Then $T$ is regular. 
\end{corollary}

\begin{proof}

Take $p_1,\cdots,p_n\in P$ and $g\in P$ with $g\wedge p_i=e$. It is clear that $g\wedge\left(p_i p_j^{-1}\right)_\pm =e$ and therefore $g$ commutes with each $\left(p_i p_j^{-1}\right)_\pm$. Hence,
\begin{eqnarray*}
T(g)^* \tilde{T}(p_ip_j^{-1}) T(g) &=& T(g)^* T((p_ip_j^{-1})_-)^* T((p_ip_j^{-1})_+) T(g) \\
&=& T((p_ip_j^{-1})_-)^*T(g)^*  T(g) T((p_ip_j^{-1})_+) \\
&=& T((p_ip_j^{-1})_-)^* T((p_ip_j^{-1})_+) =\tilde{T}(p_ip_j^{-1}).
\end{eqnarray*}
Therefore, $[T(g)^* \tilde{T}(p_ip_j^{-1}) T(g)]=[\tilde{T}(p_ip_j^{-1})]$ and Condition (\ref{eq.right}) is satisfied. \end{proof}

For a contractive representation $T$, it would suffice to dilate it to an isometric representation. This provides an analog of \cite[Proposition 2.5.4]{DFK2014} on non-abelian lattice ordered groups.

\begin{corollary} Let $T:P\to\bh{H}$ be a contractive representation. Then $T$ is completely positive definite if and only if there exists an isometric representation $V:P\to\bh{K}$ such that $P_\mathcal{H} V(p)\big|_\mathcal{H}=T(p)$ for all $p\in P$. Such $V$ can be taken to be minimal in the sense that $\mathcal{K}=\bigvee_{p\in P} V(p)\mathcal{H}$. 

In particular, $T$ is regular if and only if there exists such isometric dilation $V$ and in addition, $P_\mathcal{H} V(p)^* V(q)\big|_\mathcal{H}=T(p)^* T(q)$ for all $p,q\in P$ with $p\wedge q=e$.
\end{corollary}

\begin{proof}

When $T:P\to\bh{H}$ is completely positive definite and its extension $S$ to $G$ has minimal unitary dilation $U:G\to\bh{L}$, let $\mathcal{K}=\bigvee_{p\in P} U(p)\mathcal{H}$. It is clear that $\mathcal{K}$ is invariant for any $U(p)$, $p\in P$. Define a map $V:P\to\bh{K}$ via $V(p)=P_{\mathcal{K}} U(p)\big|_{\mathcal{K}}$, which must be isometric due to the invariance of $\mathcal{K}$. $V$ is an isometric dilation of $T$ that satisfies $P_\mathcal{H} V(p)|_\mathcal{H} = T(p)$, and $\mathcal{K}=\bigvee_{p\in P} V(p)\mathcal{H}$. In other words, $V$ is a minimal isometric dilation of $T$. In particular, when $T$ is regular, for any $p,q\in P$ with $p\wedge q=e$
\begin{eqnarray*}
T(p)^* T(q) &=& P_\mathcal{H} U(p)^* U(q)\big|_{\mathcal{H}}\\
 &=& P_\mathcal{H} P_\mathcal{K} U(p)^* U(q)\big|_{\mathcal{K}}\big|_{\mathcal{H}} \\
&=& P_\mathcal{H} V(p)^* V(q)\big|_\mathcal{H}.
\end{eqnarray*}
Conversely, when $V:P\to\bh{K}$ is a minimal isometric dilation of $T$, Corollary \ref{cor.iso} implies that $V$ is regular and thus completely positive definite. There exists a unitary dilation $U:G\to\bh{L}$ where $P_\mathcal{K} U(p)\big|_{\mathcal{K}}=V(p)$. Therefore, 
\begin{eqnarray*}
P_\mathcal{H} U(p)\big|_{\mathcal{H}} &=& P_\mathcal{H}P_\mathcal{K} U(p)\big|_{\mathcal{H}} \\
&=& P_\mathcal{H} V(p) \big|_{\mathcal{H}} = T(p).
\end{eqnarray*}
Hence, $U$ is also a unitary dilation of $T$ and thus $T$ is completely positive definite. Moreover, when $P_\mathcal{H} V(p)^* V(q)\big|_\mathcal{H}=T(p)^* T(q)$ for all $p,q\in P$ with $p\wedge q=e$, by the regularity of $V$,
$$P_\mathcal{H} U(p)^* U(q)\big|_\mathcal{H}=P_\mathcal{H} P_\mathcal{K}U(p)^* U(q)\big|_\mathcal{K}\big|_\mathcal{H}=T(p)^* T(q).$$
Therefore, $\tilde{T}(g)=T(g_-)^* T(g_+)$ is completely positive definite and $T$ is regular.
\end{proof}

\section{Nica-covariant Representations}

In this section, we answer the question of whether contractive Nica-covariant representations are regular. It suffices to show contractive Nica-covariant representations on lattice ordered groups satisfy Condition (\ref{eq.right}). 

\begin{theorem}\label{thm.nc} A contractive Nica-covariant representation on a lattice ordered group is regular.
\end{theorem}

\begin{proof}

Let $p_1,\cdots,p_k\in P$ and $g\in P$ with $g\wedge p_i=e$ for all $i=1,2,\cdots,k$.  $X=[\tilde{T}(p_i p_j^{-1})]$ and $D=\operatorname{diag}(T(g),T(g),\cdots,T(g))$. By Remark \ref{rm.pos}, we may assume $X\geq 0$.

Since for each $p_i,p_j\in P$, $\tilde{T}(p_ip_j^{-1})=T(p_{i,j}^-)^* T(p_{i,j}^+)$ where $e\leq p_{i,j}^\pm\leq p_i,p_j$. Hence, $g\wedge p_{i,j}^\pm=e$ and thus $g$ commutes with $p_{i,j}^\pm$. Therefore $T(g)$ commutes with $T(p_{i,j}^+)$ because $T$ is a representation and it also commutes with $T(p_{i,j}^-)^*$ by the Nica-covariant condition. As a result, $T(g)$ commutes with each entry in $X$, and thus $D$ commutes with $X$. Similarly, $D^*$ commutes with $X$ as well.

By continuous functional calculus, since $X\geq 0$, we know $D,D^*$ also commutes with $X^{1/2}$. Hence, in such case, 
$$D^*XD=D^*X^{1/2}X^{1/2} D=X^{1/2} D^* D X^{1/2}\leq X. \qedhere$$ \end{proof}
It was shown in \cite[Proposition 2.5.10]{DFK2014} that a contractive Nica-covariant representation on abelian lattice ordered groups can be dilated to an isometric Nica-covariant representation. Here, we shall extend this result to non-abelian case.

\begin{corollary}\label{cor.NicaIso} Any minimal isometric dilation $V:P\to\bh{K}$ of a contractive Nica-covariant representation $T:P\to\bh{H}$ is also Nica-covariant.
\end{corollary} 
\begin{proof} Let $T:P\to\bh{H}$ be a contractive Nica-covariant representation. Theorem \ref{thm.nc} implies that $T$ is regular, and thus by Theorem \ref{thm.NagyFoias}, it has a minimal unitary dilation $U:G\to\bh{L}$, which gives rise to a minimal isometric dilation $V:P\to\bh{K}$. Here $\mathcal{K}=\bigvee_{p\in P} V(p)\mathcal{H}$ and $V(p)=P_\mathcal{K} U(p)|_\mathcal{K}$. Notice that $\mathcal{K}$ is invariant for $U$ and therefore, $P_\mathcal{K}U(p)^* U(q)|_\mathcal{K}=V(p)^* V(q)$ for any $p,q\in P$. In particular, if $p\wedge q=e$, $p,q\in P$, we have from the regularity that
\begin{eqnarray*}
T(p)^* T(q) &=& P_\mathcal{H} U(p)^* U(q)|_\mathcal{H} \\
&=& P_\mathcal{H} (P_\mathcal{K}U(p)^* U(q)|_\mathcal{K}) |_\mathcal{H} \\
&=& P_\mathcal{H} V(p)^* V(q) |_\mathcal{H}.
\end{eqnarray*}
Now let $s,t\in P$ be such that $s\wedge t=e$. First, we shall prove $V(s)^* V(t)|_\mathcal{H}=V(t) V(s)^*|_\mathcal{H}$: Since $\{V(p) h: p\in P, h\in\mathcal{H}\}$ is dense in $\mathcal{K}$, it suffices to show for any $h,k\in\mathcal{H}$ and $p\in P$, $$\left\langle V(s)^* V(t) h, V(p) k\right\rangle=\left\langle  V(t) V(s)^* h, V(p) k\right\rangle.$$
Start from the left, 
\begin{eqnarray*}
& & \left\langle V(s)^* V(t) h, V(p) k\right\rangle \\
&=& \left\langle V(p)^* V(s)^* V(t) h, k\right\rangle = \left\langle V(sp)^* V(t) h, k\right\rangle \\
&=& \left\langle V((sp\wedge t)^{-1} sp)^* V(sp\wedge t)^* V(sp\wedge t) V((sp\wedge t)^{-1}t) h,k\right\rangle \\
&=& \left\langle V((sp\wedge t)^{-1} sp)^* V((sp\wedge t)^{-1}t) h, k\right\rangle \\
&=& \left\langle  T((sp\wedge t)^{-1} sp)^* T((sp\wedge t)^{-1}t) h, k\right\rangle.
\end{eqnarray*}
The last equality follows from $\left((sp\wedge t)^{-1} sp\right)\wedge\left((sp\wedge t)^{-1}t\right)=e$ and thus, 
$$T((sp\wedge t)^{-1} sp)^* T((sp\wedge t)^{-1}t) =  P_\mathcal{H} V((sp\wedge t)^{-1} sp)^* V((sp\wedge t)^{-1}t) |_\mathcal{H}.$$

Since $s\wedge t=e$, Lemma \ref{lm.1b} implies that $sp\wedge t=p\wedge t$. Notice $(p\wedge t)\wedge s\leq t\wedge s=e$, and thus by Property (4) of Lemma \ref{lm.LatticeBasic}, $s$ commutes with $p\wedge t$. By the Nica-covariance of $T$, this also implies $T(s)^*$ commutes with $T((p\wedge t)^{-1}t)$. Put all these back to the equation:
\begin{eqnarray*}
& & \left\langle  T((sp\wedge t)^{-1} sp)^* T((sp\wedge t)^{-1}t) h, k\right\rangle \\
&=& \left\langle  T( s(p\wedge t)^{-1}p)^*T((p\wedge t)^{-1}t) h, k\right\rangle \\
&=& \left\langle  T((p\wedge t)^{-1}p)^*T(s)^* T((p\wedge t)^{-1}t)  h, k\right\rangle \\
&=& \left\langle T((p\wedge t)^{-1}p)^* T((p\wedge t)^{-1}t)   \left(T(s)^* h\right), k\right\rangle \\
&=& \left\langle V((p\wedge t)^{-1}p)^* V((p\wedge t)^{-1}t)   \left(T(s)^* h\right), k\right\rangle \\
&=& \left\langle V((p\wedge t)^{-1}p)^* V((p\wedge t)^{-1}t)   \left(V(s)^* h\right), k\right\rangle \\
&=& \left\langle V(p)^* V(t)   \left(V(s)^* h\right), k\right\rangle =\left\langle V(t) V(s)^* h, V(p) k\right\rangle.
\end{eqnarray*}
Here we used the fact that $P_\mathcal{H} V(p)^* V(q) |_\mathcal{H}=T(p)^* T(q)$ whenever $p\wedge q=e$. Also, that $\mathcal{H}$ is invariant under $V(s)^*$, so that $T(s)^* h\in\mathcal{K}$ is the same as $V(s)^* h$. 

Now to show $V(s)^* V(t)=V(t) V(s)^*$ in general, it suffices to show for every $p\in P$, $V(s)^* V(t) V(p)|_\mathcal{H}=V(t) V(s)^* V(p)|_\mathcal{H}$. Start with the left hand side and repeatedly use similar argument as above,
\begin{eqnarray*}
& & V(s)^* V(t) V(p)|_\mathcal{H} \\
&=& V(s)^* V_{tp}|_\mathcal{H} =  V((s\wedge tp)^{-1}s)^* V((s\wedge tp)^{-1}tp)|_\mathcal{H} \\
&=&  V(t(s\wedge p)^{-1}p) V((s\wedge p)^{-1}s)^* |_\mathcal{H} \\
&=& V(t (s\wedge p)^{-1}p) V((s\wedge p)^{-1}s)^* |_\mathcal{H} \\
&=& V(t) V((s\wedge p)^{-1}s)^* V((s\wedge p)^{-1}p)  |_\mathcal{H} = V(t) V(s)^* V(p)|_\mathcal{H}.
\end{eqnarray*}
This finishes the proof. \end{proof}

\section{Row and Column Contractions}

A commuting $n$-tuple $(T_1,\cdots,T_n)$ where each $T_i\in\bh{H}$ is called a \emph{row contraction} if $\sum_{i=1}^n T_i T_i^* \leq I$. Equivalently, the operator $[T_1,T_2,\cdots,T_n]\in\mathcal{B}(\mathcal{H}^n,\mathcal{H})$ is contractive. It can be naturally associated with a contractive representation $T:\mathbb{Z}_+^n\to\bh{H}$ that sends the $i$-th generator $e_i$ to $T_i$. There is a dual definition called column contractions, when $T_i$ satisfies $\sum_{i=1}^n T_i^* T_i \leq I$. It is clear that $T$ is a row contraction if and only if $T^*$ is a column contraction. 

As an immediate corollary to Brehmer's theorem (Theorem \ref{thm.Brehmer}), a column contraction $T$ is always right regular \cite[Proposition I.9.2]{SFBook}, and therefore a row contraction $T$ is always left regular. This section generalizes the notion of row contraction to arbitrary lattice ordered groups and establishes a similar result.

\begin{definition}\label{def.rc} Let $T:P\to\bh{H}$ be a contractive representation of a lattice ordered group $(G,P)$. $T$ is called \emph{row contractive} if for any $p_1,\cdots,p_n\in P$ where $p_i\neq e$ and $p_i\wedge p_j=e$ for all $i\neq j$, $$\sum_{i=1}^n  T(p_i)T(p_i)^*\leq I.$$
Dually, $T$ is called \emph{column contractive} if for such $p_i$, $$\sum_{i=1}^n  T(p_i)^*T(p_i)\leq I.$$
\end{definition}

\begin{remark} Definition \ref{def.rc} indeed generalized the notion of commuting row contractions: when the group is $(\mathbb{Z}^\Omega,\mathbb{Z}_+^\Omega)$ where $\Omega$ is countable, a representation $T:\mathbb{Z}_+^\Omega\to\bh{H}$ is uniquely determined by its value on the generators $T_\omega=T(e_\omega)$. $T$ is called commuting row contraction when $\sum_{\omega\in\Omega} T_\omega T_\omega^* \leq I$. For any $p_1,\cdots,p_k\in\mathbb{Z}_+^\Omega$ where $p_i\wedge p_j=0$ for all $i\neq j$ and $p_i\neq 0$, each $p_i$ can be seen as a function from $\Omega$ to $\mathbb{Z}_+$ with finite support. Let $S_i\subseteq\Omega$ be the support of $p_i$, which is non-empty since $p_i\neq 0$. We have $S_i\bigcap S_j=\emptyset$ since $p_i\wedge p_j=0$. Therefore, pick any $\omega_i\in S_i$ and by $T$ contractive, $T(\omega_i) T(\omega_i)^*\geq T(p_i) T(p_i)^*$. Since $S_i$ are pairwise-disjoint, $\omega_i$ are distinct. Therefore, we get that
$$\sum_{i=1}^n T(p_i) T(p_i)^* \leq \sum_{i=1}^n T(\omega_i) T(\omega_i)^* \leq I.$$
and thus $T$ satisfies the Definition \ref{def.rc}. Hence, two definitions coincides on $(\mathbb{Z}^\Omega,\mathbb{Z}_+^\Omega)$.
\end{remark}

Our goal is to prove the following result:

\begin{theorem} A column contractive representation is right regular. Therefore, a row contractive representation is left regular.
\end{theorem}

We shall proceed with a method similar to the proof of Theorem \ref{thm.main}. 

\begin{lemma}\label{lm.rcbase} Let $T$ be a column contractive representation. Let $p_1,\cdots,p_n\in P$ and $g_1,\cdots,g_k\in P$ where $p_i\wedge p_{i'}=p_i\wedge g_j=g_j\wedge g_{j'}=e$ for all $1\leq i\neq i'\leq n$ and $1\leq j\neq j'\leq k$. Moreover, assume that $g_i\neq e$. Denote $X=[\tilde{T}(p_i p_j^{-1})]$ and $D_i=\diag(T(g_i),\cdots,T(g_i))$. Then, $$\sum_{i=1}^k D_i^* X D_i\leq X.$$
\end{lemma}

\begin{proof} 

The statement is clearly true for all $k$ when $n=1$. Now assuming it is true for all $k$ whenever $n<N$, and consider the case when $n=N$:

It is clear that when all of the $p_i$ are equal to $e$, then $X-\sum_{i=1}^k D_i^* X D_i$ is a $n\times n$ matrix whose entries are all equal to $I-\sum_{i=1}^k T(g_i)^* T(g_i)\geq 0$, and thus the statement is true. Otherwise, we may assume without loss of generality that $p_1\neq e$. Let $q_1=e$ and $q_2=p_2,\cdots,q_n=p_n$. Denote $X_0=[\tilde{T}(q_i q_j^{-1})]$ and $E=\diag(I,T(p_1),\cdots,T(p_1))$ be a $n\times n$ block diagonal matrix. 

Denote $Y=[\tilde{T}(p_i p_j^{-1})]_{2\leq i,j\leq n}$ and set $E_i=\diag(T(g_i),\cdots,T(g_i))$ be a $(n-1)\times(n-1)$ block diagonal matrix. Finally, set $E_{k+1}=\diag(T(p_1),\cdots,T(p_1))$ be a $(n-1)\times(n-1)$ block diagonal matrix.

From the proof of Theorem \ref{thm.main}, $$X=E^* X_0 E +\begin{bmatrix} 0 & 0 \\ 0 & Y-E_{k+1}^* Y E_{k+1} \end{bmatrix}.$$
Now $Y$ is a matrix of smaller size and thus by induction hypothesis, $$\sum_{i=1}^{k+1} E_i^* Y E_i\leq Y.$$
Hence,
\begin{eqnarray*}
Y-E_{k+1}^* Y E_{k+1} &\geq& \sum_{i=1}^k E_i^* Y E_i \\
&\geq& \sum_{i=1}^k E_i^* (Y-E_{k+1}^* Y E_{k+1}) E_i.
\end{eqnarray*}
Also notice that $E$ commutes with $D_i$ and therefore, if $\sum_{i=1}^k D_i^* X_0 D_i\leq X_0$, we have
\begin{eqnarray*}
& &\sum_{i=1}^k D_i^* X D_i \\
&=& E^*\left(\sum_{i=1}^k D_i^* X_0 D_i\right)E+\begin{bmatrix} 0 & 0 \\ 0 & \sum_{i=1}^k E_i^*(Y-E_{k+1}^* Y E_{k+1})E_i \end{bmatrix} \\
&\leq& E^* X_0 E +\begin{bmatrix} 0 & 0 \\ 0 & Y-E_{k+1}^* Y E_{k+1} \end{bmatrix} = X.
\end{eqnarray*}
Hence, $\sum_{i=1}^k D_i^* X D_i\leq X$ if $\sum_{i=1}^k D_i^* X_0 D_i\leq X_0$. This reduction from $X$ to $X_0$ changes one $p_i\neq e$ to $e$, and therefore by repeating this process, we eventually reach a state where all $p_i=e$. \end{proof}

The main result can be deduced immediately from the following Proposition:

\begin{proposition} Let $T$ be a column contractive representation on a lattice ordered semigroup $P$. Let $p_1,\cdots,p_n\in P$ and $g_1,\cdots,g_k\in P$ where $g_i\wedge p_j=e$ and $g_i\wedge g_l=e$ for all $i\leq l$. Assuming $g_i\neq e$ and denote $X=[\tilde{T}(p_i p_j^{-1})]$ and $D_i=\diag(T(g_i),\cdots,T(g_i))$. Then $$\sum_{i=1}^k D_i^* X D_i\leq X.$$ In particular, Condition (\ref{eq.right}) is satisfied when $k=1$.
\end{proposition}

\begin{proof}

The statement is clear when $n=1$. Assuming it's true for $n<N$, and consider the case when $n=N$: Let $m$ be the smallest integer such that for all $J\subseteq \{1,2,\cdots, N\}$ and $|J|>m$, $\wedge_{j\in J} p_j=e$. It was observed in the proof of Theorem \ref{thm.main} that $m\leq N-1$. Proceed by induction on $m$:

In the base case when $m=1$, $p_i\wedge p_j=e$ for all $i\neq j$, the statement is shown in Lemma \ref{lm.rcbase}. Assuming the statement is true for $m<M-1<N-1$ and consider the case when $m=M$. For each $J\subseteq\{1,2,\cdots,N\}$ with $|J|=M$ and $\wedge_{j=1}^M p_j=g\neq e$, denote $q_i=p_i$ when $i\notin J$ and $q_i= q_i g^{-1}$ when $i\in J$. Let $X_0=[\tilde{T}(q_i q_j^{-1})]$ and $E$ be a block diagonal matrix whose $i$-th diagonal entry is $I$ when $i\notin J$ and $T(g)$ otherwise. Denote $Y=[\tilde{T}(q_i q_j^{-1})]_{i,j\notin J}$ and $E_i=\diag(T(g_i),\cdots,T(g_i))$ with $N-M$ copies of $T(g_i)$. Finally, let $E_{k+1}=\diag(T(g),\cdots,T(g))$ with $N-M$ copies of $T(g)$. 

From the proof of Theorem \ref{thm.main}, by assuming without loss of generality that $J=\{1,2,\cdots,M\}$, we have
$$X=E^* X_0 E+\begin{bmatrix} 0 & 0 \\ 0 & Y-E_{k+1}^* Y E_{k+1} \end{bmatrix}.$$
Now $Y$ has a smaller size and thus by induction hypothesis on $n$, 
$$\sum_{i=1}^{k+1} E_i^* Y E_i\leq Y.$$
and thus
\begin{eqnarray*}
Y-E_{k+1}^* Y E_{k+1} &\geq& \sum_{i=1}^k E_i^* Y E_i \\
&\geq& \sum_{i=1}^k E_i^* (Y-E_{k+1}^* Y E_{k+1}) E_i.
\end{eqnarray*}
Therefore, if $\sum_{i=1}^k D_i^* X_0 D_i\leq X_0$, 
\begin{eqnarray*}
& & \sum_{i=1}^k D_i^* X D_i \\
&=& E^*\left(\sum_{i=1}^k D_i^* X_0 D_i\right)E+\begin{bmatrix} 0 & 0 \\ 0 & \sum_{i=1}^k E_i^*(Y-E_{k+1}^* Y E_{k+1})E_i \end{bmatrix} \\
&\leq& E^* X_0 E +\begin{bmatrix} 0 & 0 \\ 0 & Y-E_{k+1}^* Y E_{k+1} \end{bmatrix} = X.
\end{eqnarray*}
Hence, the statement is true for $p_i$ if it is true for $q_i$, where $\wedge_{j\in J} q_j=e$. Repeat the process until all such $|J|=M$ has $\wedge_{j\in J} p_j=e$, which reduces to a case where $m<M$. This finishes the induction. Notice Condition (\ref{eq.right}) is clearly true when $g=e$, and when $g\neq e$, it is shown by the case when $m=1$. This finishes the proof. \end{proof}

\section{Brehmer's Condition}

Brehmer \cite{Brehmer1961} established a necessary and sufficient condition for a representation on $P=\mathbb{Z}_+^\Omega$ to be regular (see Theorem \ref{thm.Brehmer}). This section explores how Brehmer's result relates to Condition (\ref{eq.right}) without invoking their equivalence to regularity. In particular, we show that Brehmer's condition allows us to decompose certain $X=[\tilde{T}(p_i-p_j)]$ as a product $R^* R$, where $R$ is an upper triangular matrix. 

Let $\{T_\omega\}_{\omega\in\Omega}$ be a family of commuting contractions, which leads to a contractive representation on $\mathbb{Z}_+^\Omega$ by sending each $e_\omega$ to $T_\omega$. For each $U\subseteq\Omega$, denote 
$$Z_U=\sum_{V\subseteq U} (-1)^{|V|} T(e_V)^* T(e_V).$$
For example, 
\begin{eqnarray*}
Z_{\emptyset} &=& I \\
Z_{\{1\}} &=& I-T_1^* T_1 \\
Z_{\{1,2\}} &=& Z_{\{1\}}-T_2^* Z_{\{1\}} T_2 = I-T_1^*T_1-T_2^*T_2+T_2^*T_1^*T_1T_2 \\
&\vdots&
\end{eqnarray*}
Brehmer's theorem stated that $T$ is regular if and only if $Z_U\geq 0$ for any finite subset $U\subseteq\Omega$. We shall first transform Brehmer's condition into an equivalent form. 
\begin{lemma}\label{lm.brehmer1} $Z_U\geq 0$ for each finite subset $U\subseteq \Omega$ if and only if for any finite set $J\subseteq\Omega$ and $\omega\in\Omega$, $\omega\notin J$, $$T_\omega^* Z_J T_\omega\leq Z_J.$$
\end{lemma}
\begin{proof}
Take any finite subset $J\subseteq\Omega$ and $\omega\in\Omega$, $\omega\notin J$. 
\begin{eqnarray*}
& & Z_J-T_\omega^* Z_J T_\omega\\
&=& \sum_{V\subseteq J} (-1)^{|V|} T(e_V)^* T(e_V) + \sum_{V\subseteq J} (-1)^{|V|+1} T_\omega^*T(e_V)^* T(e_V)T_\omega \\
&=& \sum_{V\subseteq \{\omega\}\bigcup J, \omega\notin V} (-1)^{|V|} T(e_V)^* T(e_V) + \sum_{V\subseteq \{\omega\}\bigcup J, \omega\in V} (-1)^{|V|} T(e_V)^* T(e_V) \\
&=& Z_{\{\omega\}\bigcup J}.
\end{eqnarray*}
Therefore, $T_\omega^* Z_J T_\omega\leq Z_J$ if and only if $Z_{\{\omega\}\bigcup J}\geq 0$. This finishes the proof.\end{proof}

A major tool is the following version of Douglas Lemma \cite{DouglasLemma}:

\begin{lemma}[Douglas]\label{lm.Douglas} For $A,B\in\bh{H}$, $A^*A\leq B^*B$ if and only if there exists a contraction $C$ such that $A=CB$.
\end{lemma}

As an immediate consequence of Lemma \ref{lm.Douglas}, $T_\omega^* Z_J T_\omega\leq Z_J$ is satisfied if and only if there is a contraction $W_{\omega,J}$ such that $Z_J^{1/2} T_\omega=W_{\omega,J} Z_J^{1/2}$. Therefore, it would suffices to find such contraction $W_{\omega,J}$ for each finite subset $J\subseteq\Omega$ and $\omega\in\Omega$, $\omega\notin J$. By symmetry, it would suffices to do so for each $J_n=\{1,2,\cdots,n\}$ and $\omega_n=n+1$. Without loss of generality, we shall assume that $\Omega=\mathbb{N}$. 

Consider $\mathcal{P}(J_n)=\{U\subseteq J_n\}$, and denote $p_U=\sum_{i\in U} e_i\in\mathbb{Z}_+^\Omega$. Denote $X_n=[\tilde{T}(p_U-p_V)]$ where $U$ is the row index and $V$ is the column index.

\begin{lemma}\label{lm.telescope} Assuming $Z_J\geq 0$ for all $J\subseteq J_n$. Then for a fixed $F\subseteq J_n$, we have, $$\sum_{U\subseteq F} T_U^* Z_{F\backslash U} T_U = I.$$
\end{lemma} 

\begin{proof}

We first notice that by definition, $Z_J=\sum_{U\subseteq J} (-1)^{|U|} T_U^* T_U$. Therefore,
$$\sum_{U\subseteq F} T_U^* Z_{F\backslash U} T_U = \sum_{U\subseteq F} \sum_{V\subseteq F\backslash U} (-1)^{|V|} T_{U\bigcup V}^* T_{U\bigcup V}.$$
For a fixed set $W\subseteq F$, consider the coefficient of $T_W^* T_W$ in the double summation. It appears in the expansion of every $T_U^* Z_{F\backslash U} T_U$, where $U\subseteq W$, and its coefficient in the expansion of such term is equal to $(-1)^{|W\backslash U|}$. Therefore, the coefficient of $T_W^* T_W$ is equal to $$\sum_{U\subseteq W} (-1)^{|W\backslash U|}=\sum_{i=0}^{|W|} {|W|\choose i} (-1)^i.$$
This evaluates to $0$ when $|W|>0$ and $1$ when $|W|=0$, in which case, $W=\emptyset$ and $T_W=I$. \end{proof}

Now can now decompose $X_n=R_n^* R_n$ explicitly.

\begin{proposition}\label{prop.decomp} Assuming $Z_J\geq 0$ for all $J\subseteq J_n$. Define a block matrix $R_n$, whose rows and columns are indexed by $\mathcal{P}(J_n)$, by $R_n(U,V)=Z_{J_n\backslash U}^{1/2} T_{U\backslash V}$ whenever $V\subseteq U$ and $0$ otherwise. Then $X_n=R_n^* R_n$
\end{proposition} 
\begin{proof} Fix $U,V\subseteq J_n$, the $(U,V)$-entry in $X_n$ is $\tilde{T}(p_U-p_V)=T_{V\backslash U}^* T_{U\backslash V}$. Now the $(U,V)$-entry in $R_n^* R_n$ is equal to
$$\sum_{W\subseteq J_n} R_n(W,U)^* R_n(W,V).$$
It follows from the definition that $R_n(W,U)^* R_n(W,V)=0$ unless $U,V\subseteq W$, and thus $U\bigcup V\subseteq W$. Hence,
\begin{eqnarray*}
& &\sum_{W\in\mathcal{P}(J_n)} R_n(W,U)^* R_n(W,V) \\
&=& \sum_{U\bigcup V\subseteq W} T_{W\backslash U}^* Z_{J_n\backslash W} T_{W\backslash V} \\
&=& \sum_{U\bigcup V\subseteq W} T_{V\backslash U}^* T_{W\backslash(U\bigcup V)}^* Z_{J_n\backslash W} T_{W\backslash(U\bigcup V)} T_{W\backslash U} \\
&=& T_{V\backslash U}^*  \left(\sum_{U\bigcup V\subseteq W} T_{W\backslash(U\bigcup V)}^* Z_{J_n\backslash W} T_{W\backslash(U\bigcup V)}\right) T_{W\backslash U}.
\end{eqnarray*} 
If we denote $F=J_n\backslash(U\bigcup V)$ and $W'=W\backslash (U\bigcup V)$, since $U\bigcup V\subseteq W$, we have $J_n\backslash W=F\backslash W'$. Hence the summation becomes $$\sum_{U\bigcup V\subseteq W} T_{W\backslash(U\bigcup V)}^* Z_{J_n\backslash W} T_{W\backslash(U\bigcup V)}=\sum_{W'\subseteq F} T_{W'}^* Z_{F\backslash W'} T_{W'},$$ which by Lemma \ref{lm.telescope} is equal to $I$. Therefore, the $(U,V)$-entry in $R_n^* R_n$ is equal to $T_{V\backslash U}^* T_{W\backslash U}$ and $X_n=R_n^* R_n$
\end{proof}

\begin{remark} If we order the subsets of $J_n$ by cardinality and put larger sets first, then since $R_n(U,V)\neq 0$ only when $V\subseteq U$, $R_n$ becomes a lower triangular matrix. In particular, the row of $\emptyset$ contains exactly one non-zero entry, which is $Z_{J_n}^{1/2}$ at $(\emptyset,\emptyset)$. 
\end{remark}

\begin{example} Let us consider the case when $n=2$, and $J_2$ has 4 subsets $\{1,2\}$, $\{2\}$,$\{1\}$,$\emptyset$. Under this ordering, $$X_n=\begin{bmatrix} I & T_1 & T_2 & T_1T_2 \\ T_1^* & I & T_1^* T_2 & T_2 \\ T_2^* & T_2^* T_1 & I & T_1 \\ T_1^* T_2^* & T_2^* & T_1^* & I \end{bmatrix}.$$ Proposition \ref{prop.decomp} gives that $$R_n=\begin{bmatrix} I & T_1 & T_2 & T_1 T_2 \\ 0 & Z_1^{1/2} & 0 & Z_1^{1/2} T_2 \\ 0 & 0 & Z_2^{1/2}  & Z_2^{1/2}  T_1 \\ 0 & 0 & 0 & Z_{1,2}^{1/2} \end{bmatrix}$$ satisfies $R_n^* R_n=X_n$.
\end{example}
We can now prove Brehmer's condition from Condition (\ref{eq.right}) without invoking their equivalence to regularity.

\begin{proposition} In the case of $T:\mathbb{Z}_+^\Omega\to\bh{H}$, Condition (\ref{eq.right}) implies Brehmer's condition. 
\end{proposition} 

\begin{proof} Without loss of generality, we may assume $\Omega=\mathbb{N}$. We shall proceed by induction on the size of $J\subseteq \mathbb{N}$.

For $|J|=1$ (i.e. $J=\{\omega\}$), Condition (\ref{eq.right}) implies $T$ is contractive. Hence, $Z_{J}=I-T_\omega^* T_\omega\geq 0$. Assuming $Z_{J}\geq 0$ for all $|J|\leq n$, and consider the case when $|J|=n+1$. By symmetry, it would suffices to show this for $J=J_{n+1}=\{1,2,\cdots,n+1\}$. 

By Proposition \ref{prop.decomp}, $X_n=R_n^* R_n$ where the $(\emptyset,\emptyset)$-entry of $R_n$ is equal to $Z_{J_n}^{1/2}$. Let $D_n$ be a block diagonal matrix with $2^n$ copies of $T_{n+1}$ along the diagonal. Condition (\ref{eq.right}) implies that $$D_n^* X_n D_n =D_n^* R_n^* R_n D_n\leq X_n=R_n^* R_n.$$ Therefore, by Lemma \ref{lm.Douglas}, there exists a contraction $W_n$ such that $W_n R_n=R_n D_n$. By comparing the $(\emptyset,\emptyset)$-entry on both sides, there exists $C_n$ such that $C_n Z_{J_n}^{1/2}=Z_{J_n}^{1/2} T_{n+1}$, where $C_n$ is the $(\emptyset,\emptyset)$-entry of $W_n$, which must be contractive as well. Hence, by Lemma \ref{lm.brehmer1} and \ref{lm.Douglas}, $$Z_{J_{n+1}}=Z_{J_n}-T_{n+1}^* Z_{J_n}T_{n+1}\geq 0.$$ This finishes the proof.
\end{proof} 

\section{Covariant Representations}

The semicrossed products of a dynamical system by Nica-covariant representations was discussed in \cite{Fuller2013, DFK2014}, where its regularity is seen as a key to many results. Our result on the regularity of Nica-covariant representations (Theorem \ref{thm.nc} and Corollary \ref{cor.NicaIso}) allows us to generalize some of the results to arbitrary lattice ordered abelian groups. 

\begin{definition} A $C^*$-dynamical system is a triple $(A,\alpha,P)$ where
\begin{enumerate}
\item $A$ is a $C^*$-algebra;
\item $\alpha:P\to\operatorname{End}(A)$ maps each $p\in P$ to a $*$-endomorphism on $A$;
\item $P$ is a spanning cone of some group $G$.
\end{enumerate}
\end{definition}

\begin{definition} A pair $(\pi,T)$ is called a \emph{covariant pair} for a $C^*$-dynamical system if
\begin{enumerate}
\item $\pi: A\to\bh{H}$ is a $*$-representation;
\item $T:P\to\bh{H}$ is a contractive representation of $P$;
\item $\pi(a) T(s)=T(s) \pi(\alpha_s(a))$ for all $s\in P$ and $a\in A$.
\end{enumerate} 
In particular, a covariant pair $(\pi,T)$ is called Nica-covariant/isometric, if $T$ is Nica-covariant/isometric.
\end{definition}

The main goal is to prove that Nica-covariant pairs on $C^*$-dynamical systems can be lifted to isometric Nica-covariant pairs. This can be seen from \cite[Theorem 4.1.2]{DFK2014} and Corollary \ref{cor.NicaIso}. However, we shall present a slightly different approach by taking the advantage of the structure of lattice ordered abelian group.

\begin{theorem}\label{thm.semi} Let $(A,\alpha,P)$ be a $C^*$-dynamical system over a positive cone $P$ of a lattice ordered abelian group $G$. Let $\pi:A\to\bh{H}$ and $T:P\to\bh{H}$ form a Nica-covariant pair $(\pi,T)$ for this $C^*$-dynamical system. If $V:P\to\mathcal{K}$ is a minimal isometric dilation of $T$, then there is an isometric Nica-covariant pair $(\rho, V)$ such that for all $a\in A$,
$$P_\mathcal{H}\rho(a)\big|_{\mathcal{H}}=\pi(a).$$
Moreover, $\mathcal{H}$ is invariant for $\rho(a)$.
\end{theorem}

\begin{proof}

Fix a minimal dilation $V$ of $T$ and consider any $h\in\mathcal{H}$, $p\in P$, and $a\in A$: define $$\rho(a) V(p) h = V(p) \pi(\alpha_p(a)) h$$ We shall first show that this is a well defined map. First of all, since $V$ is a minimal isometric dilation, the set $\{V(p) h\}$ is dense in $\mathcal{K}$. Suppose $V(p) h_1=V(s) h_2$ for some $p,s\in P$ and $h_1,h_2\in\mathcal{H}$. It suffices to show that for any $t\in P$ and $h\in\mathcal{H}$, we have 
\begin{equation}
\left\langle V(p) \pi(\alpha_p(a)) h_1, V(t) h\right\rangle=\left\langle V(s) \pi(\alpha_s(a)) h_2, V(t) h\right\rangle.
\end{equation}
Since $A$ is a $C^*$-dynamical system, it follows from the covariant condition $\pi(a) T(s)=T(s)\pi(\alpha_s(a))$ that $T(s)^* \pi(a)=\pi(\alpha_s(a)) T(s)^*$. Hence, 
\begin{eqnarray*}
& & \left\langle V(p) \pi(\alpha_p(a)) h_1, V(t) h\right\rangle \\
&=& \left\langle V(t)^*V(p) \pi(\alpha_p(a)) h_1, h\right\rangle \\
&=& \left\langle V(t-t\wedge p)^*V(p-t\wedge p) \pi(\alpha_p(a)) h_1, h\right\rangle \\
&=& \left\langle T(t-t\wedge p)^*T(p-t\wedge p) \pi(\alpha_p(a)) h_1, h\right\rangle \\
&=& \left\langle \pi(\alpha_{p-(p-t\wedge p)+(t-t\wedge p)}(a)) T(t-t\wedge p)^* T(p-t\wedge p)h_1, h\right\rangle \\
&=& \left\langle \pi(\alpha_{t}(a)) T(t-t\wedge p)^* T(p-t\wedge p)h_1, h\right\rangle.
\end{eqnarray*}
Here we used that fact that $V$ is regular and thus $$P_\mathcal{H} V(t-t\wedge p)^* V(p-t\wedge p)\big|_\mathcal{H}=T(t-t\wedge p)^* T(p-t\wedge p).$$ Now notice that 
\begin{eqnarray*}
T(t-t\wedge p)^* T(p-t\wedge p)h_1 &=& P_\mathcal{H} V(t-t\wedge p)^* V(p-t\wedge p) h_1 \\
&=& P_\mathcal{H} V(t)^* V(p) h_1.
\end{eqnarray*}
Similarly, 
$$\left\langle V(s) \pi(\alpha_s(a)) h_2, V(t) h\right\rangle= \left\langle \pi(\alpha_{t}(a)) T(t-t\wedge s)^* T(s-t\wedge s)h_2, h\right\rangle,$$
where $$T(t-t\wedge s)^* T(s-t\wedge s)h_2= P_\mathcal{H} V(t)^* V(s) h_2=P_\mathcal{H} V(t)^* V(p) h_1.$$
Therefore, $\rho$ is well defined on the dense subset $\{V(p) h\}$. 

Since $V(p)$ is isometric and $\pi,\alpha$ are completely contractive, $$\|V(p) \pi(\alpha_p(a)) h\|=\|\pi(\alpha_p(a)) h\|\leq\|h\|=\|V(p) h\|,$$
and thus $\rho(a)$ is contractive on $\{V(p) h\}$. Hence, $\rho(a)$ can be extended to a contractive map on $\mathcal{K}$. Moreover, for any $h\in\mathcal{H}$ and $a\in A$, we have $\rho(a) h = \pi(a) h\in\mathcal{H}$, and thus $\mathcal{H}$ is invariant for $\rho$. For any $a,b\in A$, $p\in P$, and $h\in\mathcal{H}$, 
\begin{eqnarray*}
\rho(a)\rho(b) V(p) h &=&  V(p) \pi(\alpha_p(a))\pi(\alpha_p(b)) h \\
&=& V(p) \pi(\alpha_p(ab)) h \\
&=& \rho(ab) V(p) h.
\end{eqnarray*}
Therefore, $\rho$ is a contractive representation of $A$ and thus a $*$-representation. Now for any $p,t\in P$ and $h\in\mathcal{H}$, 
\begin{eqnarray*}
\rho(a)V(p) V(t)h &=& V(p+t) \pi(\alpha_{p+t}(a))h \\
&=& V(p) V(t) \rho(\alpha_{p+t}(a))h\\
&=& V(p) \rho(\alpha_p(a)) V(t) h.
\end{eqnarray*}
Hence, $(\rho,V)$ is an isometric Nica-covariant pair. \end{proof}

This lifting of contractive Nica-covariant pairs to isometric Nica-covariant pairs has significant implication in its associated semi-crossed product. A family of covariant pairs gives rise to a semi-crossed product algebra in the following way \cite{Fuller2013, DFK2014}. For a $C^*$-dynamical system $(A,\alpha,P)$, denote $\mathcal{P}(A,P)$ be the algebra of all formal polynomials $q$ of the form
$$q=\sum_{i=1}^n e_{p_i} a_{p_i},$$
where $p_i\in P$ and $a_{p_i}\in A$. The multiplication on such polynomials follows the rule that $a e_s = e_s \alpha(a)$ and $e_p e_q=e_{pq}$. For a covariant pair $(\sigma,T)$ on this dynamical system, define a representation of $\mathcal{P}(A,P)$ by $$(\sigma\times T)\left(\sum_{i=1}^n e_{p_i} a_{p_i}\right)=\sum_{i=1}^n T(p_i) \sigma(a_{p_i}).$$
Now let $\mathcal{F}$ be a family of covariant pairs on this dynamical system. We may define a norm on $\mathcal{P}(A,S)$ by
$$\|p\|_\mathcal{F}=\sup\{(\sigma\times T)(p):(\sigma,T)\in\mathcal{F}\},$$
and the semi-crossed product algebra is defined as 
$$A\times_\alpha^\mathcal{F} P=\overline{\mathcal{P}(A,S)}^{\|\cdot\|_\mathcal{F}}.$$
In particular, $A \times_\alpha^{nc} P$ is determined by the Nica-covariant representations, and $A \times_\alpha^{nc,iso} P$ is determined by the isometric Nica-covariant representation. As an immediate corollary from Theorem \ref{thm.main} and \ref{thm.semi},

\begin{corollary} For a $C^*$-dynamical system $(A,\alpha,P)$, the semi-crossed product algebra given by Nica-covariant pairs agrees with that given by isometric Nica-covariant pairs. In other words, $$A \times_\alpha^{nc} P \cong A \times_\alpha^{nc,iso} P.$$
\end{corollary}

\section*{Acknowledgements}

I would like to thank Professor Ken Davidson for pointing out this area of research and giving me directions. I would also like to thank Adam Fuller and Evgenios Kakariadis for many valuable comments.

\bibliographystyle{plain}
\bibliography{NicaCovPaper}

\end{document}